\documentclass[11pt,a4paper]{article}

\usepackage{graphicx}%pictures
\usepackage{amsmath}
\usepackage{amssymb}
\usepackage{amsthm}
\usepackage{geometry}
\usepackage{fancyhdr}%Heads
\usepackage{color} %colors
\usepackage{overpic}%pictures with text
\usepackage{mathabx}
\usepackage[all]{xy}%commutative graph
\usepackage{tikz-cd}
\usepackage{mathrsfs}
\usepackage{amsfonts}
\usepackage{mathtools}
\usepackage[colorlinks,allcolors=black]{hyperref}
\usepackage{authblk}
\usepackage[english]{babel}
\theoremstyle{plain}
\newtheorem{thm}{Theorem}[section] 
\theoremstyle{definition}
\newtheorem{defn}[thm]{Definition}

\newtheorem{prop}[thm]{Proposition}
\newtheorem{cor}[thm]{Corollary}
\newtheorem{lem}[thm]{Lemma}
\newtheorem{rem}[thm]{Remark}
\numberwithin{equation}{section}
\title{Uniform Bogomolov Conjecture for Algebraic Tori}
\author{Ruida Di}
\date{}
\newcommand{\Addresses}{{% additional braces for segregating \footnotesize
  \bigskip
  \footnotesize

  \textsc{CAS, Morningside Ctr. Math., Beijing. Morningside Center of Mathematics and Chinese Academy of Sciences. 55 Zhongguancun E Rd, Beijing 100190, China.}\par\nopagebreak
  \textit{E-mail address:\ }\texttt{drdmath@amss.ac.cn}
}}

\begin{document}
\maketitle
\begin{abstract}
In this paper, we present a novel proof of the uniform Bogomolov conjecture for algebraic tori. To do this, we introduce a definition of non-degenerate subvarieties applicable to a family of algebraic tori and  establish an equidistribution theorem in this setting. Our method stems from recent advancements in the uniform Mordell-Lang conjecture, particularly the breakthrough results by Dimitrov, Gao, and Habegger, as well as independent work by Kühne and the collaborative effort of Gao, Ge, and Kühne. 
\end{abstract}
\tableofcontents
\section{Introduction}
This paper studies a central problem in Diophantine geometry: the uniform Bogomolov conjecture for algebraic tori. While this conjecture was resolved by Bombieri--Zannier~\cite{BZ95}, we present a novel approach to this theorem. 

Let \(\mathbf{G}_m^n\) denote the \(n\)-dimensional algebraic torus and \((\mathbb{P}^1)^n\) the product of projective lines, both defined over \(\overline{\mathbb{Q}}\). Consider the canonical embedding

\begin{equation}
    \phi: \mathbf{G}_m^n \hookrightarrow (\mathbb{P}^1)^n, \quad \phi(x_1,\ldots,x_n) = ([1:x_1],\ldots,[1:x_n]),
\end{equation}
which identifies \(\mathbf{G}_m^n\) as a Zariski open subset of \((\mathbb{P}^1)^n\).

For any subvariety \(Z \subset \mathbf{G}_m^n\), define its \textit{stabilizer} as
\begin{equation}\label{Stab}
    \mathrm{Stab}(Z) := \{a \in \mathbf{G}_m^n : aZ = Z\},
\end{equation}
and let \(Z_{\mathrm{ex}}\) denote the \textit{exceptional locus}:
\begin{equation}
    Z_{\mathrm{ex}} := \bigcup_{\substack{gH \subset Z \\ \dim H \geq 1}} gH,
\end{equation}
where the union ranges over all translates \(gH\) of positive-dimensional subtori \(H \subset \mathbf{G}_m^n\) contained in \(Z\). The \textit{essential locus} is defined as \(Z^\circ := Z \setminus Z_{\mathrm{ex}}\).

Equip \(\mathbf{G}_m^n\) with the standard height function
\begin{equation}
    \hat{h}(x) := \sum_{i=1}^n h(x_i),
\end{equation}
where \(h\) denotes the absolute logarithmic Weil height on \(\mathbb{P}^1(\overline{\mathbb{Q}})\).

The uniform Bogomolov conjecture for algebraic tori, established by Bombieri--Zannier~\cite{BZ95}, asserts the following quantitative control on small points:

\begin{thm}\label{1.1.1}
There exist constants \(c_1(d,n) > 0\) and \(c_2(d,n) > 0\) such that for every irreducible subvariety \(Z \subset \mathbf{G}_m^n\) with \(\deg_{\mathcal{O}(1,\ldots,1)} \overline{Z} = d\), 
\begin{equation}
    \#\left\{x \in Z^\circ(\overline{\mathbb{Q}}) : \hat{h}(x) \leq c_2\right\} < c_1.
\end{equation}
\end{thm}

\noindent Here, \(\overline{Z}\) denotes the Zariski closure of \(Z\) in \((\mathbb{P}^1)^n\), and the degree is computed with respect to the \(\mathcal{O}(1,\ldots,1)\) line bundle. The constants \(c_1\) and \(c_2\) depend on \(d\) and \(n\), but are independent of other geometric data of \(Z\).

\subsection{Some history of the Bogomolov conjecture}

The Bogomolov conjecture was extensively studied in 1990s, concerns the distribution of rational points on abelian varieties over number fields. The conjecture asserts that rational points exhibit "sparsity" when compared to the density of algebraic points on the variety, offering profound insights into the interplay between arithmetic and geometric structures. 

The Bogomolov conjecture for algebraic tori was proved by S.Zhang \cite{Zhang95a}. Later, Y.Bilu \cite{B97} gives an elegant proof for the algebraic tori case with a critical idea called "equidistribution of small points" introduced by L.Szpiro, E.Ullmo, S.Zhang \cite{SUZ97}. This idea, rooted in arithmetic dynamics and metrized line bundles, demonstrated how Galois orbits of points with sufficiently small heights equidistribute on analytic spaces. Building on these ideas, Bombieri and Zannier \cite{BZ95} established the uniform Bogomolov conjecture for algebraic tori (Theorem \ref{1.1.1}). During the same time, by the work of Amoroso, David, Philippon and Schmidt \cite{AD03}, \cite{AD06}, \cite{DP99}, \cite{S96}, they obtained quantitative lower bounds for the height.

The equidistribution framework also proved instrumental in Zhang’s resolution of the conjecture for abelian varieties \cite{Zha98}. However, extending the uniform Bogomolov conjecture to abelian varieties remained an open problem for years. A central obstacle was the lack of a robust equidistribution theory for families of abelian varieties, where variations in moduli and complex structures introduce technical subtleties. Recent breakthroughs by Lars Kühne \cite{LK21}, which is based on an important height inequality by Dimitrov-Gao-Habegger \cite{DGH21}, overcome this barrier and establish an equidistribution theorem for the abelian scheme.

By integrating these tools with methods from the uniform Mordell-Lang conjecture (Dimitrov-Gao-Habegger \cite{DGH21}, Gao-Ge-Kühne \cite{Gao21}, Kühne \cite{LK21}), they finalized the proof of the uniform Bogomolov conjecture for abelian varieties. Subsequent work by Xinyi Yuan and Shou-Wu Zhang \cite{Yuan21a}, \cite{Yuan21b} has further generalized these results, strengthening the gap principle and providing alternative perspectives.

\subsection{Ideas of our proof}

We begin by developing the notion of \textbf{non-degenerate subvarieties} for a family of algebraic tori. This concept parallels the corresponding theory for abelian varieties established by Habegger \cite{H13}, further developed by Gao--Habegger \cite{GH19}, and extensively studied by Gao \cite{Gao20}. It plays a critical role in the proof of uniform Mordell-Lang conjecture. Our definition serves as a toric analogue, capturing an intrinsic "bigness" property analogous to that of line bundles. Specifically, a subvariety is \textit{non-degenerate} if the associated line bundle induces a nontrivial semi-positive measure on the subvariety.

A key construction of non-degenerate subvarieties utilizes Hilbert schemes. Specifically, consider a universal family of integral subvarieties of an algebraic torus with fixed dimension and degree. By taking sufficiently high fiber products of this family with itself, we obtain a non-degenerate subvariety. This construction provides essential geometric objects for our subsequent analysis.

The concept of non-degeneracy naturally prepares for the application of equidistribution techniques for small points. Following the foundational work of Szpiro--Ullmo--Zhang \cite{SUZ97}, we establish an equidistribution theorem (Theorem \ref{431}) in the context of toric families. This result forms the technical core for our main applications.

Combining these ingredients, we prove the following  theorem for small points on subvarieties of algebraic tori, subsequently deriving Theorem~\ref{1.1.1} via a straightforward induction argument.

\begin{thm}\label{1.1.2}
 Let $r\geq 1$ and $d\geq 1$ be two integers. There exists constants $c_1(r,d,n)>0$ and $c_2(r,d,n)>0$ with the following property: 
 
 For every irreducible subvariety $Z\subset \mathbf{G}_m^n$ satisfying
 
 (1) $\dim Z=r$, $\deg_{\mathcal{O}(1,...,1)}\overline{Z}=d$;

 (2) $Z$ generates $\mathbf{G}_m^n$;

 (3) $\text{Stab}(Z)$ is finite.

There exists a proper closed subvariety $X\subset Z$ with $\deg_{\mathcal{O}(1,...,1)}(\overline{X})< c_1$ such that
$\Sigma:=\{x\in Z(\overline{\mathbb{Q}}): \hat{h}(x)\leq c_2\}$ is contained in $X(\overline{\mathbb{Q}})$.
 \end{thm}
 Here we say a subvariety $Z$ of $\mathbf{G}_m^n$ generates $\mathbf{G}_m^n$ if the smallest torus contains $ZZ^{-1}$ is $\mathbf{G}_m^n$. Equivalently, $Z$ is not contained in the translate of any proper subtorus.

 The proof strategy is to apply the equidistribution theorem twice. Our approach incorporates recent refinements by K\"{u}hne~\cite{LK21} and Gao--Ge--K\"{u}hne~\cite{Gao21}, utilizing uniform equidistribution estimates to achieve the stated uniformity in the constants.

\textbf{Acknowledgement}. The author mainly did this work in his ph.d studying in Leibniz University Hannover. The author received funding from the European Research Council
(ERC) under the European Union’s Horizon 2020 research and innovation program (grant agreement n◦ 945714) during his ph.d studying. I would like to express my gratitude to my supervisor Ziyang Gao for weekly discussion and valuable suggestion for writing. I would like also thanks Professor Thomas Gauthier and Professor Lars K\"{u}hne for pointing out mistakes in the early version of this paper and valuable discussions. 

\section{Non-degenerate subvarieties}\label{nonds}

\subsection{Non-degenerate subvarieties of products of two projective spaces}\label{DN}
Let \(\mathbb{P}^n\) and \(\mathbb{P}^w\) be projective spaces defined over \(\overline{\mathbb{Q}}\), and let \(S \subset \mathbb{P}^w\) be an arbitrary subvariety (not necessarily irreducible). Let \(x = [x_0:\cdots:x_n]\) and \(z = [z_0:\cdots:z_w]\) denote homogeneous coordinates on \(\mathbb{P}^n\) and \(\mathbb{P}^w\), respectively.

Consider a subvariety \(X \subset \mathbb{P}^n \times \mathbb{P}^w\) with projection maps
\[
p: \mathbb{P}^n \times \mathbb{P}^w \to \mathbb{P}^n, \quad \pi: \mathbb{P}^n \times \mathbb{P}^w \to \mathbb{P}^w.
\]
This yields a commutative diagram:
\[
\begin{tikzcd}
X \arrow[d, hook] \arrow[r, hook] & \mathbb{P}^n \times \mathbb{P}^w \arrow[d, "\pi"] \arrow[r, "p"] & \mathbb{P}^n \\
S \arrow[r, hook] & \mathbb{P}^w &
\end{tikzcd}
\]

Let \(L_1 = \mathcal{O}(1)\) denote the hyperplane line bundle on \(\mathbb{P}^n\). Define the logarithmic potentials \(f_l\) and associated \((1,1)\)-forms \(T_l\) for \(l \in \mathbb{Z}_+ \cup \{\infty\}\):
\begin{align}
    f_l &:= \log\left( \frac{|x_0|^l + \cdots + |x_n|^l}{n+1} \right)^{1/l}, \quad l \in \mathbb{Z}_+, \label{3.3.2} \\
    f_\infty &:= \log\left( \max\{ |x_0|, \ldots, |x_n| \} \right), \label{3.3.3} \\
    T_l &:= dd^c f_l. \label{3.3.4}
\end{align}
Here \(T_2\) coincides with the Fubini-Study \((1,1)\)-form on \(\mathbb{P}^n\), while all \(T_l\) represent the first Chern class \(c_1(L_1)\). 

Define the line bundle \(L := p^*L_1 = \mathcal{O}(1,0)\) on \(\mathbb{P}^n \times \mathbb{P}^w\) and its associated \((1,1)\)-forms
\[
U_l := p^*T_l = dd^c f_l.
\]
Clearly, each \(U_l\) represents \(c_1(L)\), and for any \(s \in \mathbb{P}^w\), the restriction \(U_l|_{\pi^{-1}(s)}\) coincides with \(T_l\) under the identification \(\pi^{-1}(s) \cong \mathbb{P}^n\). In particular, \(U_2|_{\pi^{-1}(s)}\) equals the Fubini-Study form.

\begin{defn}\label{331}
The subvariety \(X\) is \emph{non-degenerate over \(S\)} if it satisfies any of the following equivalent conditions:
\begin{enumerate}
    \item[(1)] The line bundle \(L|_X\) is big.
    
    \item[(2)] There exists \(l \in \mathbb{Z}_+ \cup \{\infty\}\) such that \(U_l|_X^{\wedge \dim X} \neq 0\) as a measure on \(X(\mathbb{C})\).
    
    \item[(3)] For all \(l \in \mathbb{Z}_+ \cup \{\infty\}\), the measure \(U_l|_X^{\wedge \dim X}\) is nonzero on \(X(\mathbb{C})\).
    
    \item[(4)] There exists a smooth point \(x \in X^{\mathrm{sm}}(\mathbb{C})\) where the differential 
    \[
    \mathrm{rank}_{\mathbb{C}}(dp|_{X^{\mathrm{sm,an}}})_x = \dim X.
    \]
\end{enumerate}
\end{defn}

\begin{proof}[Proof of equivalence in Definition~\ref{331}]
For \(l \in \mathbb{Z}_+\), the form \(U_l\) represents \(c_1(L)\). Thus 
\[
\int_{X(\mathbb{C})} U_l^{\wedge \dim X} = \deg_L(X).
\]
By Bedford-Taylor theory (Corollary 3.6, \cite{D12}), \(U_l \to U_\infty\) as currents. The equivalence of (1)--(3) follows from the characterization of bigness via volume computations ((2.9), \cite{R.L}). 

To establish (2) \(\Leftrightarrow\) (4), we only need to consider the case \(l=2\).

If (4) is false, then every $x\in X^{\text{sm,an}}$ is a non-isolated point of $p^{-1}(r)\cap X^{\text{sm,an}}$ where $r=p(x)$. As $p^{-1}(r)\simeq \mathbb{P}^{w}$ is a complex analytic subvariety and $X^{\text{sm,an}}$ is complex analytic in a neighborhood of $x$ in $\mathbb
{P}^{n}\times \mathbb{P}^{w}$, there exists an irreducible complex analytic curve $C$ in  $p^{-1}(r)\cap X^{\text{sm,an}}$ passing through $x$. From the definition of $U_2$, $U_2|_{C^{\text{sm}}}\equiv 0$ (because $U_2$ is independent of the last factor). This gives $(U_2|_{X}^{\wedge\text{dim}X})_{x}=0$ for every point $x\in C^{\text{sm}}$. By continuity, it also holds for $x\in C$.
Since $x$ is taken arbitrary, $U_2|_{X^{\text{sm}}}^{\wedge\text{dim}X}\equiv 0$.

Conversely, suppose $U_2|_{X^{\text{sm}}}^{\wedge\text{dim}X}\equiv 0$, we can find an irreducible, 1-dimensional complex analytic subset $C_{x}\subset X^{\text{sm,an}}$ passing through $x$ and $U_2$ vanishes along $C_{x}^{\text{sm}}$. From the definition of $U_2$, it is clear that $p(C_{x})$ is a point. Hence $\text{rank}_{\mathbb{C}}(dp|_{X^{\text{sm,an}}})_x$ is strictly less than $\text{dim}(X)$.

\end{proof}

\subsection{Non-degenerate subvarieties of multi-projective spaces}\label{2.2.2}
The notion of non-degenerate subvarieties extends naturally to products of multiple projective spaces. Let \(X \subset \mathbb{P}^{n_1} \times \cdots \times \mathbb{P}^{n_t} \times \mathbb{P}^w\) be a subvariety over a base \(S \subset \mathbb{P}^w\).

Let \(\sigma: \mathbb{P}^{n_1} \times \cdots \times \mathbb{P}^{n_t} \times \mathbb{P}^w \hookrightarrow \mathbb{P}^N\) denote the Segre embedding. Define projections:
\begin{align}
    p_i&: \mathbb{P}^{n_1} \times \cdots \times \mathbb{P}^{n_t} \times \mathbb{P}^w \to \mathbb{P}^{n_i}, \quad i \in \{1,\ldots,t\}, \\
    \pi&: \mathbb{P}^{n_1} \times \cdots \times \mathbb{P}^{n_t} \times \mathbb{P}^w \to \mathbb{P}^w,
\end{align}
and let \(p: \mathbb{P}^{n_1} \times \cdots \times \mathbb{P}^{n_t} \times \mathbb{P}^w \to \mathbb{P}^N\) be the composition:
\[
\begin{tikzcd}[column sep=small]
(\mathbb{P}^{n_1} \times \cdots \times \mathbb{P}^{n_t}) \times \mathbb{P}^w \arrow[rr, "\sigma \times \mathrm{id}"] && \mathbb{P}^N \times \mathbb{P}^w \arrow[r] & \mathbb{P}^N.
\end{tikzcd}
\]

Equip the space with the line bundle \(L := \mathcal{O}(1,\ldots,1,0)\) and define:
\begin{itemize}
    \item \(U_L\): a \((1,1)\)-form representing \(c_1(L)\)
    \item \(U_\infty\): pullback of \(T_\infty := dd^c \log \max\{|x_0|,\ldots,|x_N|\}\) via \(p\)
\end{itemize}

\begin{defn}\label{333}
The subvariety \(X\) is \emph{non-degenerate over \(S\)} if it satisfies any of:
\begin{enumerate}
    \item[(1)] \(L|_X\) is a big line bundle
    \item[(2)] \((U_L|_{X}^{\wedge \dim X})_x \neq 0\) for some \(x \in X^{\mathrm{sm}}(\mathbb{C})\)
    \item[(3)] \(U_\infty|_{X}^{\wedge \dim X} \neq 0\) as a measure on \(X(\mathbb{C})\)
    \item[(4)] \(\mathrm{rank}_{\mathbb{C}}(dp|_{X^{\mathrm{sm,an}}})_x = \dim X\) for some \(x \in X^{\mathrm{sm}}(\mathbb{C})\)
\end{enumerate}
\end{defn}

\begin{proof}[Proof of equivalence]
Analogous to Definition~\ref{331}.
\end{proof}

Explicitly, let \([x_0^{(i)}:\cdots:x_{n_i}^{(i)}]\) denote homogeneous coordinates on \(\mathbb{P}^{n_i}\). Define \((1,1)\)-forms for \(l \in \mathbb{Z}_+\):
\[
U_l := \sum_{i=1}^t dd^c\left( \log\left( \frac{|x_0^{(i)}|^l + \cdots + |x_{n_i}^{(i)}|^l}{n_i+1} \right)^{1/l} \right),
\]
which represent \(c_1(L)\). The limit case \(U_\infty\) is given by:
\[
U_\infty := \sum_{i=1}^t dd^c\left( \log \max\{ |x_0^{(i)}|, \ldots, |x_{n_i}^{(i)}| \} \right).
\]
By Bedford-Taylor theory, \(U_l \to U_\infty\) as currents when \(l \to \infty\).

\subsection{A criterion for non-degeneracy}
To operationalize Definition~\ref{333}, we introduce degeneracy loci. 

\begin{defn}\label{334}
Let \(X \subset \mathbb{P}^{n_1} \times \cdots \times \mathbb{P}^{n_t} \times \mathbb{P}^w\) be as in \S\ref{2.2.2}.
The \emph{degeneracy locus} of \(X\) over \(S\) is defined as:
\[
X^{\mathrm{deg}} := \bigcup_{\substack{\{t\} \times C \subset X}} (\{t\} \times C),
\]
where $t\in \mathbb{P}^{n_1}\times...\times \mathbb{P}^{n_t}$ and $C$ runs over curves in $\mathbb{P}^{w}$ that satisfy $\{t\}\times C\subset X$.
\end{defn}

\begin{prop}\label{335}
If \(X \setminus X^{\mathrm{deg}}\) contains a nonempty analytic open subset, then \(X\) is non-degenerate.
\end{prop}
\begin{proof}
Let \(U \subset X \setminus X^{\mathrm{deg}}\) be such an open set. For any \(x \in U\), the fiber \(p^{-1}(p(x)) \cap X\) contains no curves by construction. Thus \(\mathrm{rank}_{\mathbb{C}}(dp|_{X^{\mathrm{sm,an}}})_x = \dim X\). 
\end{proof}

\begin{prop}\label{nonc}
Let \(X \subset \mathbb{P}^{n_1} \times \mathbb{P}^w\) and \(Y \subset \mathbb{P}^{n_2} \times \mathbb{P}^w\) be irreducible \(S\)-subvarieties with dominant projections to \(S\subset \mathbb{P}^{w}\). If \(X\) is non-degenerate, then \(X \times_S Y \subset \mathbb{P}^{n_1} \times \mathbb{P}^{n_2} \times \mathbb{P}^w\) is non-degenerate over \(S\).
\end{prop}
\begin{proof}
Denote by
\begin{itemize}
    \item $\pi:\mathbb{P}^{n_1}\times \mathbb{P}^{n_2}\times \mathbb{P}^{w}\rightarrow \mathbb{P}^{n_1}\times\mathbb{P}^{n_2}$: the projection to first two factors;
    \item $p_{1}:\mathbb{P}^{n_1}\times \mathbb{P}^{w} \rightarrow \mathbb{P}^{n_1}$: the projection map to $\mathbb{P}^{n_1}$;
    \item $p_{2}: \mathbb{P}^{n_2}\times \mathbb{P}^{w} \rightarrow \mathbb{P}^{n_2}$: the projection map to $\mathbb{P}^{n_2}$.
\end{itemize}
 
By generic smoothness, We may assume $S$ is smooth, and $\pi_{X}:X\rightarrow S$ and $\pi_{Y}:Y\rightarrow S$ are smooth morphisms.
The proposition follows from the calculation of fiber dimension:

We have $\dim X\times_{S} Y=\dim X+\dim Y-\dim S$.  Non-degeneracy of \(X\) provides \(x \in X^{\mathrm{sm}}(\mathbb{C})\) with \(\mathrm{rank}_{\mathbb{C}}(dp|_{X^{\mathrm{sm,an}}})_x = \dim X\). Let \(s=\pi(x)\).

As $\pi|_{\pi^{-1}(s)}$ is an isomorphism to $\mathbb{P}^{n_1}\times\mathbb{P}^{n_2}$, we have 
$$\text{rank}_{\mathbb{C}}(dp|_{Y^{\text{sm,an}}})_{y}\geq \dim Y_{s}=\dim Y-\dim S$$
for generic $y\in Y_{s}^{\text{sm}}(\mathbb{C})$. Then $y\in Y^{\text{sm}}(\mathbb{C})$ as $S$ is smooth and $Y^{\text{sm}}\rightarrow S$ is smooth.

Take $(x,y)\in (X\times_{S} Y)$ above, then $\text{rank}_{\mathbb{C}}(d(p_{X}\times_{S} p_{Y})|_{(X\times_{S}Y)^{\text{sm,an}}})_{(x,y)}=\dim X+\dim Y-\dim S$. Hence $X\times_{S} Y$ is non-degenerate.
\end{proof}

\subsection{Construction of a non-degenerate subvariety}\label{CND}

In this section, we construct a key example of non-degenerate subvarieties by means of Hilbert schemes.

Fix integers \( r,d \geq 1 \) and \( n \geq 1 \). Let \( \mathbf{P} := (\mathbb{P}^1)^n \) serve as our ambient variety. We identify the algebraic torus \( \mathbf{G}_m^n \) with a subvariety of \( \mathbf{P} \) via the immersion:
\[
\phi: \mathbf{G}_m^n \hookrightarrow \mathbf{P}, \quad [x_1,\ldots,x_n] \mapsto ([1:x_1],\ldots,[1:x_n]).
\]
This particular compactification is chosen due to technical advantages in describing equilibrium measures in following sections.

Fix integers $r,d\geq 1$, $n\geq 1$. Let $\mathbf{P}:=(\mathbb{P}^1)^n$. 
Identifying $\mathbf{G}_m^n$ as a subvariety of $\mathbf{P}$ by $\phi:\mathbf{G}_m^n \rightarrow \mathbf{P}$, where
\begin{equation}
    \phi([x_1,...,x_n]):= ([1:x_1],\cdot\cdot\cdot,[1:x_n]).
\end{equation}
The reason we take $\mathbf{P}$ as the compactification of $\mathbf{G}_m^n$ rather than $\mathbb{P}^n$ is for purely technical reason. Specifically, it is easier to give an explicit description of the equilibrium measure for this case.

Let \( L := \mathcal{O}(1,\ldots,1) \) denote the standard ample line bundle on \( \mathbf{P} \). By Grothendieck's fundamental results on Hilbert schemes (see \cite{Gro62}, Theorem 2.1(b) and Lemma 2.4), there are finitely many possibilities of Hilbert polynomials  corresponding to irreducible \( r \)-dimensional subvarieties of \( \mathbf{P} \) with degree \( d \) relative to \( L \). Let \(\Lambda\) be this finite set of polynomials.

By results of Grothendieck (see \cite{Gro62}, Thm 2.1(b) and Lem 2.4), there are finitely many possibilities for the Hilbert polynomials of irreducible subvarieties of $\mathbf{P}$ with dimension $r$ and degree $d$ respect to $L$. 

 We define the parameter space:
\[
\mathbf{H}_{r,d}(n) := \bigcup_{F \in \Lambda} \mathbf{H}_F(n),
\]
where \( \mathbf{H}_F(n) \) denotes the Hilbert scheme parametrizing closed subschemes of \( \mathbf{P} \) with Hilbert polynomial \( F \). This scheme is projective and of finite type over \( \overline{\mathbb{Q}} \) by \cite{Gro62} theorem 3.2.

Consider the universal family \( \mathcal{Y}_{r,d} \rightarrow \mathbf{H}_{r,d}(n) \) with its natural closed immersion:
\[
\begin{tikzcd}
\mathcal{Y}_{r,d} \arrow[r, hook] \arrow[dr] & \mathbf{P} \times \mathbf{H}_{r,d}(n) \arrow[d,"\pi"] \\
& \mathbf{H}_{r,d}(n),
\end{tikzcd}
\]
where \( \pi \) denotes the projection. For each \( s \in \mathbf{H}_{r,d}(n)(\overline{\mathbb{Q}}) \), the fiber \( (\mathcal{Y}_{r,d})_s \) corresponds to the subscheme parametrized by \( s \).

Define the quasi-projective subvariety:
\[
\mathbf{H}_{r,d}(n)^\circ := \left\{ s \in \mathbf{H}_{r,d}(n) \mid (\mathcal{Y}_{r,d})_s \text{ is an integral subscheme of } \mathbf{P} \right\}_{\text{red}},
\]
which parametrizes integral subvarieties of dimension \( r \) and degree \( d \). The restricted universal family becomes:
\[
\mathcal{Y}_{r,d}^\circ := \mathcal{Y}_{r,d} \cap (\mathbf{P} \times \mathbf{H}_{r,d}(n)^\circ).
\]
Consider 
\begin{equation}
    \mathbf{H}_{r,d}(n)^{\circ}:=\{s\in \mathbf{H}_{r,d}(n): (\mathcal{Y}_{r,d})_{s} \text{\ is an integral subscheme of\ } \mathbf{P}\}
\end{equation}
endowed with reduced subsheme structure. 

It is a quasi-projective subvariety of $\mathbf{H}_{r,d}(n)$ defined over $\overline{\mathbb{Q}}$ which parametrize all integral subvarieties of $\mathbf{P}$ with dimension $r$ and degree $d$. 

Using the irreducibility of \( \mathbf{P} \),  we establish a bijection between irreducible closed subsets of \( \mathbf{G}_m^n \) and those of \( \mathbf{P} \) intersecting \( \mathbf{G}_m^n \). This yields a Zariski open \( \mathbf{H} \subset \mathbf{H}_{r,d}(n)^\circ \) parametrizing integral subschemes \( Z \subset \mathbf{G}_m^n \) with \( \dim(Z) = r \) and \( \deg_L \overline{Z} = d \). It is clear that \( \mathbf{H} \subset \mathbb{P}^w \) is a quasi-projective subvariety. We define our key object:
\[
\mathcal{X}_{r,d} := \mathcal{Y}_{r,d}^\circ \cap (\mathbf{G}_m^n \times \mathbf{H}),
\]
which serves as the universal family over \( \mathbf{H} \), i.e. for any $s\in \mathbf{H}$, $(\mathcal{X}_{r,d})_{s}$ is the integral subvariety corresponding to $s$.

Let $p:\mathbf{P}\times \mathbb{P}^{w} \rightarrow \mathbb{P}^{w}$ and $\pi:\mathbf{P}\times \mathbb{P}^{w} \rightarrow \mathbf{P}$ be projections. For any integral subscheme $Z$ of $\mathbf{G}_{m}^n$, denote by $[Z]$ the corresponding point in $\mathbf{H}$.

\begin{prop}\label{lem1}
Let $V$ be a subvariety of $\mathbf{H}$. There exists finitely many points $P_1,...,P_{u}$ in $\mathbf{G}_{m}^n(\overline{\mathbb{Q}})$ such that the Zariski closed subset of $V$ defined by 
\begin{equation}
    \{[Z]\in V: P_1\in Z(\overline{\mathbb{Q}}), ... , P_{u}\in Z(\overline{\mathbb{Q}})\}
\end{equation}

  is a non-empty finite set.
\end{prop}

\begin{proof}
Fix $[X_0]\in V(\overline{\mathbb{Q}})$, we prove the following result by induction:

For each integer $0\leq k\leq \text{dim}(V)$, there are finitely many points  \( P_{k,1},\ldots,P_{k,n_k} \in X_0(\overline{\mathbb{Q}}) \)  such that \[
V_k := \left\{ [Z] \in V(\overline{\mathbb{Q}}) \mid P_{k,i} \in Z(\overline{\mathbb{Q}}) \right\}
\]
has dimension \( \leq \dim(V) - k \). 

We notice that the lemma follows by taking $k=\text{dim}(V)$. 

The base case \( k=0 \) is trivial. Assume the result is true for integers less than $k$. Take  \( P_{k,1},\ldots,P_{k-1,n_{k-1}} \in X_0(\overline{\mathbb{Q}}) \) such that 
\begin{equation}
    V_{k-1}:=\{[Z]\in V(\overline{\mathbb{Q}}):P_{k-1,n_i}\in Z(\overline{\mathbb{Q}}) \text{\ for all\ }i=1,...,n_k-1\}
\end{equation}
has dimension $\leq \text{dim}(V)-k+1$. If \( \dim(V_{k-1})\leq \dim(V) - k \), we already done. If not,
for each irreducible component $W$ of $V_{k-1}$ with $\text{dim}(W)>0$, there exist $[Z_{W}]\in W(\overline{\mathbb{Q}})$ such that $Z_{W}\neq X_{0}$ as irreducible subvarieties of $\mathbf{G}_{m}^n$. Take $P_{W}\in (X_{0}\backslash{Z_{W}}) (\overline{\mathbb{Q}})$. Then $\{[Z]\in W(\overline{\mathbb{Q}}): P_{W}\in Z(\overline{\mathbb{Q}})\}$ has dimension $\leq \text{dim}(W)-1\leq \text{dim}(V)-k+1-1=\text{dim}(V)-k$.
Thus it suffices to take $\{P_{k,1},...,P_{k,n_{k}}\}:=\{P_{k-1,1},...,P_{k-1,n_{k-1}}\}\bigcup (\bigcup_{W}\{P_{W}\})$, where $W$ runs over all positive dimensional irreducible components of $V_{k-1}$. In addition, this does not lead to a empty set as all $P_i$'s are contained in $X_0$.
\end{proof}

Let $u$ be a positive integer, define the \( u \)-fold fiber product:
\[
\mathcal{X}_{r,d}^{[u]} := \mathcal{X}_{r,d} \times_{\mathbf{H}} \cdots \times_{\mathbf{H}} \mathcal{X}_{r,d},
\]
with associated commutative diagram:
\[
\begin{tikzcd}
\mathcal{X}_{r,d}^{[u]} \arrow[r, hook] \arrow[dr] & (\mathbf{G}_m^n)^u \times \mathbf{H} \arrow[d,"\pi"] \arrow[r,"p^{[u]}"] & (\mathbf{G}_m^n)^u \\
& \mathbf{H}
\end{tikzcd}
\]

\begin{rem}\label{342}
Notice that if we take $S$ to be an irreducible subvariety of $\mathbf{H}$, then $\mathcal{X}_{r,d}^{[u]}|_{S}$ is also irreducible as each fiber is irreducible.
\end{rem}
\begin{prop}\label{lem2}
For any reduced subscheme \( S \subset \mathbf{H} \), there exists \( u_0 \geq 1 \) such that for all \( u \geq u_0 \), there exists \( P \in (\mathbf{G}_m^n)^u(\overline{\mathbb{Q}}) \) making \( (p^{[u]}|_{\mathcal{X}_{r,d}^{[u]} \times_{\mathbf{H}} S})^{-1}(P) \) finite and non-empty.
\end{prop}

\begin{proof}
Let $P_1,\ldots,P_{u}\in \mathbf{G}_{m}^n(\overline{\mathbb{Q}})$ that was constructed in proposition \ref{lem1}. For each $u\geq u_0$, let $P_{k}=P_{u_0}$ for all $k\geq u_0$.
Let $$P=(P_1,\ldots,P_u)\in (\mathbf{G}_{m}^n)^{u}(\overline{\mathbb{Q}}).$$

Then $(p^{[u]}|_{\mathcal{X}_{r,d}^{[u]}\times_{\mathbf{H}} S})^{-1}(P)$ is precisely the subset 
$$\{(P,[X])\in (\mathcal{X}_{r,d}^{[u]}\times_{\mathbf{H}}S)(\overline{\mathbb{Q}}): P_{i}\in Z(\overline{\mathbb{Q}}) \text{\ for each\ } i=1,\ldots, u\}.$$

This is a finite subset by proposition \ref{lem1}.
\end{proof}

\begin{prop}\label{344}
For any \( S \subset \mathbf{H} \), there exists \( u_0 \geq 0 \) such that \( \mathcal{X}_{r,d}^{[u]}|_S \setminus (\mathcal{X}_{r,d}^{[u]}|_S)^{\text{deg}} \) contains a non-empty Zariski open subset when \( u \geq u_0 \).
\end{prop}

\begin{proof}
Let $V=S$. Take $u_0$ as in proposition \ref{lem2}, there exists $P=(P_i)_{i=1}^{u}\in  (\mathbf{G}_{m}^n)^{u}(\overline{\mathbb{Q}})$ such that $(p^{[u]}|_{\mathcal{X}_{r,d}^{[u]}})^{-1}(P)$ is non-empty and finite. As
$(p^{[u]}|_{\mathcal{X}_{r,d}^{[u]}})^{-1}(P)$ is precisely 
$$\{(P,[Z])\in \mathcal{X}_{r,d}^{[u]}(\overline{\mathbb{Q}}): P_{i}\in Z(\overline{\mathbb{Q}}) \text{\ for each\ } i=1,\ldots, u\},$$

there are only finitely many points $[Z_1],\ldots, [Z_{l}]\in S $ with the property $$P_{i}\in Z(\overline{\mathbb{Q}}),\text{\ for\ }i=1,...,u.$$

However if $(P,[Z])\in (\mathcal{X}_{r,d}^{[u]}|_{S})^{\text{deg}}(\overline{\mathbb{Q}})$, then there must exists a curve $C\subset S$ such that for each $[Y]\in C$, $(P,[Y])\in \mathcal{X}_{r,d}^{[u]}(\overline{\mathbb{Q}})$. This is impossible by the definition of $P$.
This implies that no point in $(\mathcal{X}_{r,d}^{[u]}|_{S})^{\text{deg}}(\overline{\mathbb{Q}})$ has its projection to be $P\in (\mathbf{G}_{m}^n)^{u}(\overline{\mathbb{Q}})$. 

By Chevalley's theorem (\cite{Gro62} theorem 13.1.3) on semicontinuous fiber dimension, 
$$\{(P,[Z])\in (\mathcal{X}_{r,d}^{[u]}|_{S})(\overline{\mathbb{Q}}):(p^{[u]}|_{\mathcal{X}_{r,d}^{[u]}})^{-1}(P) \text{\ is finite}\} $$
is a nonempty Zariski open subset, which is disjoint from the degeneracy locus $(\mathcal{X}_{r,d}^{[u]}|_{S})^{\text{deg}}$.
\end{proof}

\begin{cor}\label{345}
\( \mathcal{X}_{r,d}^{[u]}|_S \) becomes non-degenerate when \( u \gg 0 \).
\end{cor}

\begin{proof}
Immediate from Proposition \ref{344} and the definition of non-degeneracy.
\end{proof}

\section{Equidistribution Results}\label{ER}

This section is devoted to establishing a key equidistribution result in our setting.

\subsection{Setup}\label{s31}
Let $S$ be an irreducible subvariety of $\mathbb{P}^{w}$, and denote by $\overline{S}$ its Zariski closure in $\mathbb{P}^{w}$. We work with homogeneous coordinates $[x_0:\cdots:x_n]$ on $\mathbb{P}^n$ throughout this discussion.

Consider the product space $\mathbb{P}^n \times \mathbb{P}^w$ with projection maps:
\begin{align*}
p &: \mathbb{P}^n \times \mathbb{P}^w \to \mathbb{P}^n \\
\pi &: \mathbb{P}^n \times \mathbb{P}^w \to \mathbb{P}^w
\end{align*}
These induce the following commutative diagram:
\begin{equation}
    \begin{tikzcd}
X \arrow[d, hook] \arrow[r, hook] & \mathbb{P}^n \times \mathbb{P}^{w} \arrow[d, "\pi"] \arrow[r, "p"] & \mathbb{P}^n \\
S \arrow[r, hook] & \mathbb{P}^{w} &
\end{tikzcd}
\end{equation}

Let $L = \mathcal{O}(1,0)$ denote the tautological line bundle on $\mathbb{P}^n \times \mathbb{P}^w$. We define a family of functions parameterized by $l \in \mathbb{Z} \cup \{\infty\}$:
\begin{align*}
f_l &:= \log\left( \frac{|x_0|^l + \cdots + |x_n|^l}{n+1} \right)^{1/l} \quad \text{for } l \in \mathbb{Z}_{>0} \\
f_{\infty} &:= \log\left( \max\{|x_0|, \ldots, |x_n|\} \right)
\end{align*}
These induce corresponding $(1,1)$-currents:
\begin{equation}
    U_{l} := dd^c f_{l} \quad \text{for } l \in \mathbb{Z}_{>0} \cup \{\infty\}
\end{equation}
Each $U_l$ represents the first Chern class $c_1(L)$ in cohomology.

The canonical height $\hat{h}_L$ associated to $L$ can be constructed via Tate's limiting argument as follows. For any integer $N > 1$, consider the fiberwise multiplication map:
\begin{equation}
[N] : \mathbb{P}^n \times \mathbb{P}^w \to \mathbb{P}^n \times \mathbb{P}^w \quad \text{ by } ([x_0,\ldots,x_n], z) \mapsto ([x_0^N,\ldots,x_n^N], z)
\end{equation}
The canonical height is then obtained as the limit:
\begin{equation}
    \hat{h}_L(x) = \lim_{m \to \infty} \frac{1}{N^m} h_L([N^m](x))
\end{equation}
where $h_L$ is taken to be any height function associated to the line bundle $L$. This construction restricts to the standard Weil height on each fiber $\mathbb{P}^n \times \{z\}$ for $z \in \mathbb{P}^w$.
\subsection{Statement of the Theorem} 

\begin{thm}\label{431}
Let $S$ be an irreducible subvariety of $\mathbb{P}^{w}$, and let 
\[
X\subset \mathbb{P}^{n}\times \mathbb{P}^{w}
\]
be an irreducible non-degenerate subvariety defined over $\mathbb{Q}$ (for example, $\mathcal{X}_{r,d}^{[u]}$ for sufficiently large $u$, by Proposition \ref{345}). Assume that $\{x_n\}\subset X(\overline{\mathbb{Q}})$ is a Zariski-generic sequence of small points (i.e., the sequence converges in the Zariski topology to the generic point of $X$) and that
\[
\lim_{n\rightarrow +\infty}\hat{h}_{L}(x_n)=0.
\]
Then
\begin{equation}\label{4.3.1}
\frac{1}{\#O(x_n)}\sum_{y\in O(x_n)}\delta_{y}\xrightarrow{\text{weakly}} k\,U_{\infty}^{\wedge\dim X},
\end{equation}
where 
\[
k=\left(\int_{X(\mathbb{C})}U_{\infty}|_{X}^{\wedge \dim X}\right)^{-1}.
\]
Here, $O(x_n)$ denotes the Galois orbit of $x_n$.
\end{thm}

In the case of a multi-projective space (see Definition \ref{333}), the above theorem can be modified as follows:

\begin{thm}\label{432}
Suppose that 
\[
X\subset \mathbb{P}^{n_1}\times \cdots \times \mathbb{P}^{n_t} \times \mathbb{P}^{w}
\]
is a non-degenerate subvariety defined over $\mathbb{Q}$ and lying over an irreducible subvariety $S\subset \mathbb{P}^{w}$. Denote the homogeneous coordinates on $\mathbb{P}^{n_i}$ by
\[
[x_0^{(i)}:\cdots:x_{n_i}^{(i)}],
\]
and set
\[
L^{'}:=\mathcal{O}(1,\dots,1,0).
\]
Assume that $\{x_n\}\subset X(\overline{\mathbb{Q}})$ is a Zariski-generic sequence of small points (in the sense described above). Then
\begin{equation}\label{equ}
\frac{1}{\#O(x_n)}\sum_{y\in O(x_n)}\delta_{y}\xrightarrow{\text{weakly}} k\,(U^{'}_{\infty})^{\wedge\dim X},
\end{equation}
where 
\[
U^{'}_{\infty}:=\sum_{i=1}^{t}dd^{c}\left(\log \max\{|x_0^{(i)}|,\dots,|x_{n_i}^{(i)}|\}\right)
\]
and 
\[
k=\left(\int_{X(\mathbb{C})}(U^{'}_{\infty}|_X)^{\wedge \dim X}\right)^{-1}.
\]
Here, $O(x_n)$ denotes the Galois orbit of $x_n$.
\end{thm}

The following proposition is a direct consequence of the equidistribution result above.

\begin{prop}\label{433}
Let $X$ be a non-degenerate irreducible subvariety of $\mathbb{P}^n\times \mathbb{P}^{w}$ defined over $\mathbb{Q}$, and let 
\[
\mu=k\,U_{\infty}^{\wedge \dim X}
\]
be the measure from Theorem \ref{431}. Then, for every function $f\in C^{\infty}_{c}(X^{an})$ and every $\epsilon>0$, there exist a proper subvariety $Z_{f,\epsilon}$ of $X$ and a constant $\delta_{\epsilon}>0$ such that, for every point 
\[
x\in (X\setminus Z_{f,\epsilon})(\overline{\mathbb{Q}}),
\]
either
\begin{itemize}
    \item $\hat{h}_{L}(x)\geq \delta_{\epsilon}$; 
    \item or
    \[
    \left|\frac{1}{\# O(x)}\sum_{y\in O(x)}f(y)-\int_{X^{an}}f\,\mu\right|<\epsilon.
    \]
\end{itemize}
\end{prop}

\begin{proof}
Suppose, for contradiction, that there exists a function $f\in C^{\infty}_{c}(X^{an})$ and an $\epsilon>0$ such that for every $\delta>0$ the set
\begin{equation}
D_{\delta}:=\{x\in X(\overline{\mathbb{Q}}):\hat{h}_{L}(x)<\delta \text{ and } \left|\frac{1}{\# O(x)}\sum_{y\in O(x)}f(y)-\int_{X^{an}}f\,\mu\right|\geq\epsilon\}
\end{equation}
is Zariski dense in $X$. Choose a sequence $\delta_{n}\rightarrow 0_{+}$. Since there are only countably many proper closed subvarieties of $X$ defined over $\overline{\mathbb{Q}}$, denote them by $\{Z_n\}_{n\in \mathbb{N}}$. For each $n\in \mathbb{N}$, the Zariski density of $D_{\delta_n}$ implies that we can select a point 
\[
x_n\in (D_{\delta_n}\setminus Z_{n})(\overline{\mathbb{Q}}).
\]
Then the sequence $\{x_n\}$ satisfies:
\begin{itemize}
     \item $\{x_n\}$ converges in the Zariski topology to the generic point of $X$;
    \item $\hat{h}_{L}(x_{n}) \rightarrow 0$; 
    \item and
    \[
    \left|\frac{1}{\# O(x_n)}\sum_{y\in O(x_n)}f(y)-\int_{X^{an}}f\,\mu\right|\geq\epsilon.
    \]
\end{itemize}

This contradicts Theorem \ref{431}, and the proof is complete.
\end{proof}

\begin{rem}
If 
\[
X\subset \mathbb{P}^{n_1}\times \cdots \times \mathbb{P}^{n_t}\times \mathbb{P}^{w}
\]
is a non-degenerate subvariety lying over $S\subset \mathbb{P}^{w}$, then the above proposition remains valid upon replacing $L$ with 
\[
L^{'}=\mathcal{O}(1,\dots,1,0).
\]
\end{rem}
\subsection{Results from Arakelov Geometry}
We begin by introducing essential tools from Arakelov geometry. Basic references for this section include \cite{BGS94}, \cite{GS87}, \cite{GS90}, and \cite{PR21}.

Let $\mathcal{X}$ be a proper flat scheme over $\mathbf{Z}$ with $n := \dim(\mathcal{X})$. For each integer $d \geq 0$, let $Z_d(\mathcal{X})$ denote the group of $d$-cycles on $\mathcal{X}$. 

Given Hermitian line bundles $\overline{\mathcal{L}}_1, \ldots, \overline{\mathcal{L}}_n$ on $\mathcal{X}$, there exists a unique family of linear maps:
\begin{equation}
    \widehat{\deg}(\overline{\mathcal{L}}_1 \cdots \overline{\mathcal{L}}_d | \cdot): Z_d(\mathcal{X}) \rightarrow \mathbf{R},
\end{equation}
for $d \in \{0, \ldots, n\}$, satisfying the following properties:
\begin{itemize}
    \item For any $d \in \{1, \ldots, n\}$, closed subscheme $Z \subset \mathcal{X}$ with $\dim(Z) = d$, nonzero integer $m$, and regular meromorphic section $s$ of $\mathcal{L}_d^m|_Z$:
    \begin{equation}
        m \widehat{\deg}(\overline{\mathcal{L}}_1 \cdots \overline{\mathcal{L}}_d |_{Z}) = \widehat{\deg}(\overline{\mathcal{L}}_1 \cdots \overline{\mathcal{L}}_{d-1}|_{\div(s)}) + \int_{Z(\mathbf{C})} \log \|s\|^{-1} c_1(\overline{\mathcal{L}}_1) \cdots c_1(\overline{\mathcal{L}}_{d-1}).
    \end{equation}
    
    \item For any closed point $z \in \mathcal{X}$:
    \begin{equation}
        \widehat{\deg}(z) = \log (\#\kappa(z)),
    \end{equation}
    where $\kappa(z)$ denotes the residue field of $z$.
\end{itemize}
These maps are multilinear and symmetric in the Hermitian line bundles $\overline{\mathcal{L}}_1, \ldots, \overline{\mathcal{L}}_d$.

Let $X$ be a proper scheme over $\mathbb{Q}$ with line bundle $L$. An integral model of $(X,L)$ consists of a Hermitian line bundle $\overline{\mathcal{L}}$ on $\mathcal{X}$ such that $\mathcal{X}_{\mathbb{Q}} = X$ and $\mathcal{L}_{\mathbb{Q}} = L$. For a closed integral subscheme $Z \subset X$ of dimension $d$ with Zariski closure $\mathcal{Z}$ in $\mathcal{X}$ (satisfying $\dim(\mathcal{Z}) = d+1$), the height of $Z$ relative to $\overline{\mathcal{L}}$ is defined as:
\begin{equation}
    h_{\overline{\mathcal{L}}}(Z) = \widehat{\deg}(\overline{\mathcal{L}}^{d+1}|_{\mathcal{Z}})\big/[(d+1)\deg_L(Z)].
\end{equation}
For any $x \in X(\overline{\mathbb{Q}})$, the height of $x$ with respect to $\overline{\mathcal{L}}$ is given by:
\begin{equation}
    h_{\overline{\mathcal{L}}}(x) := h_{\overline{\mathcal{L}}}(\overline{\{x\}}),
\end{equation}
where $\overline{\{x\}}$ denotes the Zariski closure of $\{x\}$ in $X$. This construction yields a well-defined representative of the height function associated with $L$.

An integral model $(\mathcal{X}, \overline{\mathcal{L}})$ of $(X,L)$ induces a canonical adelic metric on $X$ (see (1.2) in \cite{Zhang95b}). For a line bundle $\overline{L}$ endowed with an adelic metric, we define:
\begin{equation}
    \hat{H}^0(X, \overline{L}) := \{s \in H^0(X,L) : \|s\|_v \leq 1 \text{ for all places } v\},
\end{equation}
and its logarithmic count:
\begin{equation}
    \hat{h}^0(X, \overline{L}) := \log (\# \hat{H}^0(X, \overline{L})).
\end{equation}

Let $\mathbb{A}$ be the ad\`{e}le ring of $\mathbb{Q}$ and $\mu$ a Haar measure on $H^0(X,L) \otimes \mathbb{A}$. For any place $v$ of $K$ and section $s \in H^0(X,L)$, define $\|s\|_v := \sup_{x \in X(\overline{\mathbb{Q}_v})} \|s(x)\|_v$. Let $B_v$ denote the unit ball in $H^0(X,L) \otimes \mathbb{A}$. The $\chi$-character is defined as:
\begin{equation}
    \chi(X, \overline{L}) = -\log \frac{\mu(H^0(X,L) \otimes \mathbb{A}/H^0(X,L))}{\mu(\prod_v B_v)}.
\end{equation}
We subsequently define the volume and $\chi$-volume of $\overline{L}$ by:
\begin{equation}
\text{vol}(\overline{L}):=\limsup_{N\rightarrow \infty}\frac{h^0(NL)}{N^{n}/n!},
\end{equation}
 \begin{equation}
     \text{vol}_{\chi}(\overline{L}):=\limsup_{N\rightarrow \infty}\frac{\chi(N\overline{L})}{N^{n}/n!}.
 \end{equation}

\subsection{Proof of Theorem \ref{431}}
Our argument follows the strategy of Szpiro, Ullmo, and Zhang \cite{SUZ97}. 

Consider the projective space $\mathbb{P}_{\mathbb{Q}}^N$ equipped with the tautological line bundle $\mathcal{O}(1)$, endowed with the following Hermitian metric. For any point $x = [x_0:\cdots:x_N] \in \mathbb{P}^N(\mathbb{C})$ and section $\alpha = a_0x_0 + \cdots + a_Nx_N \in H^0(\mathbb{P}^N, \mathcal{O}(1))$, we define the fiber norm at $x$ by
\begin{equation}
    \|\alpha\|(x) := \frac{|a_0x_0 + \cdots + a_Nx_N|}{\max\{|x_0|, \ldots, |x_N|\}}.
\end{equation}
Denote this Hermitian line bundle by $\overline{\mathcal{O}}(1)$. Through tensor products, this metric naturally induces Hermitian structures on $\mathcal{O}(k)$ for all $k \in \mathbb{N}_+$.

For multiprojective space $\mathbb{P}^{N_1} \times \cdots \times \mathbb{P}^{N_t}$, the Hermitian metric on $\mathcal{O}(k_1, \ldots, k_t)$ arises from the isomorphism
\begin{equation}
    \mathcal{O}(k_1, \ldots, k_t) \simeq pr_1^*\mathcal{O}(k_1) \otimes \cdots \otimes pr_t^*\mathcal{O}(k_t),
\end{equation}
where $pr_i$ denotes the $i$-th projection onto $\mathbb{P}^{N_i}$. This construction yields Hermitian line bundles $\overline{\mathcal{O}}(k_1, \ldots, k_t)$ for positive integers $k_1, \ldots, k_t$.

\begin{proof}[Proof of Theorem \ref{431}]
Let $\mathcal{X}$ be the Zariski closure of $\overline{X}$ in $\mathbb{P}_{\mathbb{Z}}^n \times \mathbb{P}_{\mathbb{Z}}^w$, which decomposes as $\mathcal{X} = \coprod_{v \in S} \mathcal{X}_v$,   where $S = \{p \mid p \text{ prime}\} \cup \{\infty\}$ denotes the complete set of places of $\mathbb{Q}$. Let $\mathcal{L} = \mathcal{O}(1,0)$ extend $L$ to $\mathcal{X}$, and set $v := \dim(\mathcal{X}) = \dim(X) + 1$.

The pair $(\mathcal{X}, \overline{\mathcal{L}})$ forms an integral model of $(\overline{X}, L)$. Following \cite[§1.1-1.2]{Zhang95b}, this model induces a canonical adelic metric $\overline{L}$ on $\overline{X}$ whose restriction to the generic fiber coincides with our construction.

By definition \ref{331}, the measure $U_{\infty}^{\wedge \dim(X)}$ is non-vanishing on $X$, and $L$ constitutes a big line bundle over $X$. Recall Zhang's first successive minimum \cite[(1.9)]{Zhang95b}:
\begin{equation}
    e_1(\overline{L}) := \sup_{Y \subset X} \inf_{x \in X(\overline{K}) \setminus Y(\overline{K})} h_{\overline{\mathcal{L}}}(x),
\end{equation}
where the supremum ranges over all proper closed subvarieties $Y$ of $X$. For any generic sequence $\{x_n\}$, we immediately have
\begin{equation}
    \liminf_{n \to \infty} h_{\overline{\mathcal{L}}}(x_n) \geq e_1(\overline{L}).
\end{equation}
Since our sequence $\{x_n\}$ is generically small, it follows that $e_1(\overline{L}) = 0$.

Applying the fundamental inequality (Theorem 1.10 of \cite{Zhang95b} or Lemma 5.1 of \cite{CT06}), we obtain
\begin{equation}
    e_1(\overline{L}) \geq \frac{\text{vol}_{\chi}(\overline{L})}{v \cdot \text{vol}(L)}.
\end{equation}
Consequently,
\begin{equation}
    \text{vol}_{\chi}(\overline{L}) = 0.
\end{equation}

Through Yuan's arithmetic bigness results (Theorem 2.2 of \cite{Yuan08} or Lemma 5.2 of \cite{CT06}),
\begin{equation}
    \text{vol}_{\chi}(\overline{L}) \geq \widehat{\deg}(\overline{\mathcal{L}}^v|_{\mathcal{X}}),
\end{equation}
which forces $\widehat{\deg}(\overline{\mathcal{L}}^v|_{X}) = 0$.

Define the measure $d\mu := \frac{c_1(\overline{L})^{v-1}}{\deg_L(X)}$ on $X(\mathbb{C})$, which scales the measure induced by $U_{\infty}|_X^{\wedge \dim(X)}$. We demonstrate that for any $f \in C_c^\infty(X(\mathbb{C}))$,
\begin{equation}
    \frac{1}{\#O(x_n)} \sum_{x_n^g \in O(x_n)} f(x_n^g) \to \int_{X(\mathbb{C})} f(x) d\mu \quad \text{as } n \to \infty.
\end{equation}

Let $u_n := \frac{1}{\#O(x_n)} \sum_{x_n^g \in O(x_n)} f(x_n^g)$. Suppose toward contradiction that $u_n \nrightarrow \int_{X(\mathbb{C})} f(x) d\mu$. By boundedness of $\{u_n\}$, we may assume (passing to a subsequence if necessary) that $u_n \to C$ for some $C \neq \int_{X(\mathbb{C})} f(x) d\mu$. Replacing $f$ by $f - \int_{X(\mathbb{C})} f(x) d\mu$, we normalize to $\int_{X(\mathbb{C})} f(x) d\mu = 0$.

For $\lambda \in \mathbb{R}$, consider the modified Hermitian line bundle $\overline{\mathcal{L}}(\lambda f) = (\mathcal{L}, \|\cdot\|')$ with archimedean metric scaled by $e^{-\lambda f}$:
\begin{equation}
    \|t\|'(x) = \begin{cases}
        \|t\|(x) & x \notin \mathcal{X}_\infty(\mathbb{C}) \\
        e^{-\lambda f(x)} \|t\|(x) & x \in \mathcal{X}_\infty(\mathbb{C})
    \end{cases}.
\end{equation}
The height transforms as
\begin{equation}
    h_{\overline{\mathcal{L}}(\lambda f)}(x) = h_{\overline{\mathcal{L}}}(x) + \frac{\lambda}{\#O(x)} \sum_{x^g \in O(x)} f(x^g),
\end{equation}
yielding $\lim_{n \to \infty} h_{\overline{\mathcal{L}}(\lambda f)}(x_n) = \lambda C$.

Reapplying Zhang's fundamental inequality and Yuan's arithmetic bigness:
\begin{equation}
    \frac{\widehat{\deg}(\overline{\mathcal{L}}(\lambda f)^v|_{\mathcal{X}})}{v \cdot \text{vol}(L)} \leq \frac{\text{vol}_\chi(\overline{\mathcal{L}}(\lambda f))}{v \cdot \text{vol}(L)} \leq e_1(\overline{L}(\lambda f)) \leq \lambda C.
\end{equation}
However, expanding the arithmetic degree:
\begin{align*}
    \widehat{\deg}(\overline{\mathcal{L}}(\lambda f)^v|_{\mathcal{X}}) &= \widehat{\deg}\left((\overline{\mathcal{L}} + \overline{\mathcal{O}}(\lambda f))^v|_{\mathcal{X}}\right) \\
    &= \widehat{\deg}(\overline{\mathcal{L}}^v|_{\mathcal{X}}) + v \widehat{\deg}(\overline{\mathcal{L}}^{v-1} \cdot \overline{\mathcal{O}}(\lambda f)|_{\mathcal{X}}) + O(\lambda^2) \\
    &= \lambda \int_{X(\mathbb{C})} f c_1(L)^{v-1} +O(\lambda^2) \\
    &= O(\lambda^2).
\end{align*}
This implies $\lambda C \geq O(\lambda^2)$. Taking $\lambda \to 0^+$ and $\lambda \to 0^-$ successively forces $C = 0$, completing the proof.
\end{proof}
\section{Apply Equidistribution Results to the Uniform Bogomolov Conjecture}

The primary objective of this section is to present a novel proof of the uniform Bogomolov conjecture for algebraic tori.

\subsection{Notational Setup}
Fix integers \( n, d, r \geq 1 \). Consider the morphism \( \phi = (\phi_1, \ldots, \phi_n): \mathbf{G}_m^n \rightarrow (\mathbb{P}^1)^n \) defined by
\begin{equation}\label{501}
    \phi(x_1, \ldots, x_n) = ([1:x_1], \ldots, [1:x_n]).
\end{equation}
Recall from Section \ref{CND} that we constructed \( \mathcal{X}_{r,d} \subset (\mathbb{P}^1)^n \times \mathbb{P}^w \) as the universal family of irreducible subvarieties of \( \mathbf{G}_m^n \) with dimension \( r \) and degree \( d \) relative to \( \mathcal{O}(1, \ldots, 1) \). The parameter space for this family is a quasi-projective subvariety \( \mathbf{H} \subset \mathbb{P}^w \). Let \( \mathcal{B} := \mathbf{G}_m^n \times \mathbf{H} \) denote the corresponding family of algebraic tori.

We work with the line bundle \( L := \mathcal{O}(1, \ldots, 1, 0) \) over \( (\mathbb{P}^1)^n \times \mathbb{P}^w \), where the zero component corresponds to \( \mathbb{P}^w \) and the \( (1, \ldots, 1) \) components correspond to \( (\mathbb{P}^1)^n \). For simplicity, denote \( \mathbf{P} := (\mathbb{P}^1)^n \).

\subsection{Main Theorem}

We establish the following fundamental result:

\begin{thm}\label{512}
Let \( S \) be a subvariety of \( \mathbf{H} \). There exist positive constants \( c_1(r,d,n,S) \) and \( c_2(r,d,n,S) \) satisfying the following property:

For any \( s \in S(\overline{\mathbb{Q}}) \) such that:
\begin{enumerate}
    \item Each fiber \( (\mathcal{X}_{r,d})_s \) is an irreducible variety generating \( \mathcal{B}_s = \mathbf{G}_m^n \);
    \item The stabilizer \( \text{Stab}((\mathcal{X}_{r,d})_s) \) is finite;
\end{enumerate}
the set
\begin{equation}
    \left\{ x \in (\mathcal{X}_{r,d})_s(\overline{\mathbb{Q}}) : \hat{h}_L(x) \leq c_2 \right\}
\end{equation}
is contained in a proper closed subvariety \( Z \subset (\mathcal{X}_{r,d})_s \) with \( \deg_{L_s}(\overline{Z}) < c_1 \).
\end{thm}

When \( S = \mathbf{H} \), parameterizing all integral subvarieties of \( \mathbf{G}_m^n \), the line bundle \( L_s \) corresponds to \( \mathcal{O}(1, \ldots, 1) \) restricted to \( (\mathcal{X}_{r,d})_s \). In this context, \( \hat{h}_L(x) \) represents the standard fiberwise height. This yields the following equivalent formulation:

\begin{thm}\label{513}
Let \( r \geq 1 \) and \( d \geq 1 \) be integers. There exist positive constants \( c_1(r,d,n) \) and \( c_2(r,d,n) \) with the following property:

For any irreducible subvariety \( Z \subset \mathbf{G}_m^n \) satisfying:
\begin{enumerate}
    \item \( \dim Z = r \) and \( \deg_{\mathcal{O}(1,\ldots,1)} \overline{Z} = d \);
    \item \( Z \) generates \( \mathbf{G}_m^n \);
    \item \( \text{Stab}(Z) \) is finite;
\end{enumerate}
there exists a proper closed subvariety \( X \subset Z \) with \( \deg_{\mathcal{O}(1,\ldots,1)}(\overline{X}) < c_1 \) such that
\[
    \Sigma := \{ x \in Z(\overline{\mathbb{Q}}) : \hat{h}(x) \leq c_2 \}
\]
is contained in \( X(\overline{\mathbb{Q}}) \).
\end{thm}

The connection between Theorem \ref{513} and Theorem \ref{1.1.1} will be elucidated in Section \ref{5.3}.

\subsection{Investigation of the Equilibrium Measure}\label{5.1}

We investigate the equilibrium measure established in Theorems \ref{431} and \ref{432}.

Let $\mathcal{X}_{r,d}$ be as defined in Section \ref{CND}, and let $S$ be an irreducible component of $\mathbf{H}$. By corollary \ref{345}, we fix a positive integer $u$ such that $X := \mathcal{X}_{r,d}^{[u]}|_S$ constitutes a non-degenerate subvariety over $S$. Let $\overline{X}$ denote the Zariski closure of $X$ in $(\mathbb{P}^1)^{nu} \times \mathbb{P}^w$, where each fiber $(\mathcal{X}_{r,d}^{[u]})_s$ forms an irreducible subvariety of $(\mathbf{G}_m^n)^u \subset (\mathbb{P}^1)^{nu}$. Set $t := nu$ and $d_X := \dim X$. The non-degeneracy of $X$ implies that the projection $p \colon X \to (\mathbb{P}^1)^t$ is generically finite, hence $d_X \leq t$.

Through Theorem \ref{432}, we associate an equilibrium measure $\mu_{\overline{X}}$ to $\overline{X}$. The following commutative diagram illustrates our setting:
\begin{equation}
    \begin{tikzcd}
        X \arrow[d] \arrow[r, hook] & \mathbf{G}_m^t \times \mathbb{P}^w \arrow[r, hook] & (\mathbb{P}^1)^t \times \mathbb{P}^w \arrow[dl, "\pi"] \arrow[r, "p"] & (\mathbb{P}^1)^t \\
        S \arrow[r, hook] & \mathbb{P}^w
    \end{tikzcd},
\end{equation}
where $p$ and $\pi$ denote the respective projections.

Let $(z_0, \ldots, z_t, z)$ be coordinates on $(\mathbb{P}^1)^t \times \mathbb{P}^w$ with $z_i = [x_0^{(i)}:x_1^{(i)}] \in \mathbb{P}^1$ and $z \in \mathbb{P}^w$. The equilibrium measure $\mu_{\overline{X}}$ on $\overline{X}$ can be expressed explicitly as a multiple of
\begin{equation}\label{5522}
    \mu = \left(\sum_{i=1}^t dd^c \log \max\{|x_0^{(i)}|, |x_1^{(i)}|\}\right)^{\! \wedge d_X}.
\end{equation}
Restricting to $X$, we have
\begin{equation}\label{5523}
    \mu|_X = \left(\sum_{i=1}^t dd^c \log \max\{1, |x_1^{(i)}|\}\right)^{\! \wedge d_X}.
\end{equation}

For each $1 \leq i \leq t$, define $M_i := dd^c \log \max\{|x_0^{(i)}|, |x_1^{(i)}|\}$. Let $I$ denote the set of $d_X$-tuples $\mathbf{i} = (i_1, \ldots, i_{d_X})$ for which $M_{i_1} \wedge \cdots \wedge M_{i_{d_X}}$ does not vanish identically on $\overline{X}$. This yields the decomposition
\begin{equation}
    \mu = \binom{t}{d_X} \sum_{\mathbf{i} \in I} M_{i_1} \wedge \cdots \wedge M_{i_{d_X}}.
\end{equation}

For each $\mathbf{i} \in I$, define:
\begin{align}
    X_{\mathbf{i}} &:= \{x \in X : |\phi_{i_1}(x)| = \cdots = |\phi_{i_{d_X}}(x)| = 1\}, \\
    \phi_{\mathbf{i}} &:= (\phi_{i_1}, \ldots, \phi_{i_{d_X}}) \colon \mathbf{G}_m^t \times \mathbf{H} \to (\mathbb{P}^1)^{d_X}, \\
    E_{\mathbf{i}} &:= \{x \in X : \ker(d\phi_{\mathbf{i}}|_X) \neq \{0_x\}\} \cup \left((\mathbb{P}^1)^t \setminus \mathbf{G}_m^t\right).
\end{align}

\begin{lem}\label{5521}
    For each $\mathbf{i} \in I$:
    \begin{enumerate}
        \item The restriction $\phi_{\mathbf{i}}|_{\overline{X} \setminus E_{\mathbf{i}}}$ constitutes a local biholomorphism to $(\mathbb{P}^1)^{d_X}$;
        \item The restriction $\phi_{\mathbf{i}}|_{X_{\mathbf{i}} \setminus E_{\mathbf{i}}}$ forms a real-analytic local isomorphism to $(S^1)^{d_X}$.
    \end{enumerate}
\end{lem}

\begin{proof}
    By definition of $I$, the wedge product $M_{i_1} \wedge \cdots \wedge M_{i_{d_X}}$ does not vanish identically on $X$. Each $M_i$ represents the first Chern class of $pr_i^*\mathcal{O}(1)$, where $pr_i \colon (\mathbb{P}^1)^t \to \mathbb{P}^1$ denotes the $i$-th projection. The projection formula \cite[Theorem 2.5(c)]{WF98} ensures $\overline{X} \setminus E_{\mathbf{i}} \neq \emptyset$. As $\overline{X}$ is irreducible, $E_{\mathbf{i}}$ is a closed complex analytic subset of dimension $<d_X$.
    By definition of $E_{\mathbf{i}}$, and $\dim \overline{X}=d_X=\dim (\mathbb{P}^1)^{d_X}$, we see that $\phi_{\mathbf{i}}$ restricts on $\overline{X}\setminus E_{\mathbf{i}}$ is a local biholomorphism  map to $(\mathbb{P}^1)^{d_X}$. Hence $X_{\mathbf{i}}\setminus E_{\mathbf{i}}=\phi_{\mathbf{i}}^{-1}((S^1)^{d_X})\cap (\overline{X}\setminus E_{\mathbf{i}})$, maps to $(S^1)^{d_X}$ is real analytic local  isomorphism.

\end{proof}

Let $dw = \left(\frac{1}{2\pi}\right)^{d_X} d\theta_1 \wedge \cdots \wedge d\theta_{d_X}$ denote the normalized Haar measure on $(S^1)^{d_X}$, where $\theta_i$ parametrizes the angular coordinate on each circle factor. Define $dw_{\mathbf{i}} := \phi_{\mathbf{i}}^*(dw)$ on $X_{\mathbf{i}} \setminus E_{\mathbf{i}}$.

\begin{prop}
    For each $\mathbf{i} \in I$ and all $f \in C_c(\overline{X})$:
    \begin{equation}\label{5.2.8}
        \int_{\overline{X}} f M_{i_1} \wedge \cdots \wedge M_{i_{d_X}} = \int_{X_{\mathbf{i}} \setminus E_{\mathbf{i}}} f dw_{\mathbf{i}}.
    \end{equation}
\end{prop}

\begin{proof}
  As $E_{\mathbf{i}}$ is locally pluripolar (being a closed complex analytic subset with $\dim E_{\mathbf{i}} < d_X$), we may assume $f \in C_c(\overline{X} \setminus E_{\mathbf{i}})$. By Lemma \ref{5521}(1), $\phi_{\mathbf{i}}$ restricts on $\overline{X}\setminus E_{\mathbf{i}}$ is a local biholomorphism  map to $(\mathbb{P}^1)^{d_X}$, we may further assume $f$ has support in some open subset $U$ such that $\phi_{\mathbf{i}}$ maps $U$ biholomorphic to some open subset $(\mathbb{P}^1)^{d_X}$. Formula \ref{5.2.8} follows from the elementary calculation 
    \begin{equation}
        \int_{\mathbb{P}^1}g dd^c\log \max\{|x_0|,|x_1|\}=\int_{S^1}gdw\ \text{for\ }g\in C_c(\mathbb{P}^1) 
    \end{equation}
     combined with Fubini's theorem.

\end{proof}

This proposition establishes the measure decomposition:
\begin{equation}
    \mu(f) = \binom{t}{d_X} \sum_{\mathbf{i} \in I} \int_{X_{\mathbf{i}} \setminus E_{\mathbf{i}}} f dw_{\mathbf{i}}, \quad f \in C_c(\overline{X}).
\end{equation}
The equilibrium measure $\mu_{\overline{X}}$ is then obtained by normalizing $\mu$ such that $\mu_{\overline{X}}(\overline{X}) = 1$.

\subsection{Proof of Uniform Bogomolov Conjecture}\label{5.3}

Assuming theorem \ref{513}, we can derive theorem \ref{1.1.1} through a straightforward induction argument as follows.

\begin{proof}[Deduction of theorem \ref{1.1.1} from theorem \ref{513} (or \ref{512})]

We proceed by induction on $r = \dim Z$. 

\textbf{Base case:} For $r = 0$, the result holds trivially.

\textbf{Inductive step:} For $r > 0$, consider two cases:
\begin{itemize}
    \item If $\operatorname{Stab}(Z)$ is infinite, then $Z^{\circ} = \emptyset$ by definition, requiring no further proof.
    \item If $\operatorname{Stab}(Z)$ is finite, then $Z$ generates a subtorus $T$ of $\mathbf{G}_m^n$ with $\operatorname{rank}(T) \leq n$. Applying theorem \ref{513}, there exist constants $c_1(r,d,n)$ and $c_2(r,d,n)$ such that the set
    \begin{equation}
        \{x \in Z(\overline{\mathbb{Q}}) : \hat{h}(x) \leq c_2(r,d,n)\}
    \end{equation}
    is contained in a proper closed subvariety $Z' \subset Z$ with $\deg_L(\overline{Z'}) < c_1$. Since $Z'$ has at most $c_1$ irreducible components and $\dim Z' \leq r-1$, where $c_1$ and $c_2$ depend only on $r,d,n$, the theorem follows by induction.
\end{itemize}
\end{proof}

To complete the argument, it remains to establish Theorem \ref{512}.

\begin{proof}[Proof of Theorem \ref{512}]
We divide the proof into several strategic steps.

\textbf{Step 1: Constructing non-degenerate subvarieties.}\\ 
 Consider the set
\begin{equation}
    T := \{s \in S(\overline{\mathbb{Q}}) : s \text{ satisfies conditions (1) and (2) in Theorem \ref{512}}\}.
\end{equation}
Let $S'$ be an irreducible component of $\overline{T}$. The finite collection of such $S'$ is uniquely determined by $S$. If we can find constants $c_1(r,d,n,S') > 0$ and $c_2(r,d,n,S') > 0$ such that for any $s \in T(\overline{\mathbb{Q}})$, the set
\begin{equation}
    \{x \in (\mathcal{X}_{r,d})_s(\overline{\mathbb{Q}}) : \hat{h}_{L}(x) \leq c_2\}
\end{equation}
lies in a proper closed subvariety $Z \subset (\mathcal{X}_{r,d})_s$ with $\deg_{L_s}(\overline{Z}) < c_1$, then considering each irreducible component of $\overline{T}$ will complete the proof.

Fix an irreducible component $S'$ of $\overline{T}$. By Proposition \ref{345}, for sufficiently large $u$, the restriction $\mathcal{X}_{r,d}^{[u]}|_{S'}$ becomes a non-degenerate subvariety of $\mathcal{B}^{[u]}|_{S'}$. Fix such a $u$ and let $X := \mathcal{X}_{r,d}^{[u]}|_{S'}$, which is irreducible in $(\mathbb{P}^1)^{nu} \times \mathbb{P}^w$ since $S'$ is irreducible with irreducible fibers. Let $t := nu$.

Define the fiberwise Faltings-Zhang map:
\begin{equation}
    \alpha_{u,M} : X^{[M+1]} \rightarrow (\mathcal{B}_{S'}^{[u]})^{[M]} = \mathbf{G}_m^{tM} \times S'
\end{equation}
acting fiberwise via $(a_0, a_1, \ldots, a_M) \mapsto (a_1a_0^{-1}, \ldots, a_Ma_0^{-1})$, where operations are understood fiberwise within the rank $t$ torus. Let
\begin{equation}
    \alpha := (\operatorname{id}, \alpha_{u,M}) : X \times_{S'} X^{[M+1]} \rightarrow \mathcal{B}_{S'}^{[u(M+1)]}.
\end{equation}
By Proposition \ref{nonc}, the image $\alpha(X^{[M+2]}) = X \times_{S'} \alpha_{u,M}(X^{[M+1]})$ forms a non-degenerate subvariety of $\mathcal{B}_{S'}^{[u(M+1)]}$.

\textbf{Step 2. Compare two measures on the non-degenerate subvarieties.}

\noindent Fix an integer $M$ sufficiently large (to be determined later). Applying equidistribution theory (Theorems \ref{431} and \ref{432}) to both $X^{[M+2]}$ and its image $\alpha(X^{[M+2]})$, we let $\mu_1$ and $\mu_2$ denote their respective equilibrium measures for $X^{[M+2]}$ and $\alpha(X^{[M+2]})$.

Our strategy involves selecting $M$ large enough to ensure that $\alpha$ becomes generic finite. Let $\eta$ denote the generic point of $S'$, and consider the stabilizer subgroups:

\begin{equation}
\operatorname{Stab}_{\mathcal{B}_{\eta}}((\mathcal{X}_{r,d})_{\eta}) := \left\{ b \in \mathcal{B}_{\eta} \mid b(\mathcal{X}_{r,d})_{\eta} = (\mathcal{X}_{r,d})_{\eta} \right\}
\end{equation}

Similarly, we define the unipotent stabilizer subgroup:
\begin{equation}
\operatorname{Stab}_{\mathcal{B}_{\eta}}(X_{\eta}) := \left\{ b \in \mathcal{B}_{\eta}^{u} \mid bX_{\eta} = X_{\eta} \right\}
\end{equation}

These definitions are well-posed since $\mathcal{B}_{\eta} = \mathbf{G}_m^{n}$ and $(\mathcal{X}_{r,d})_{\eta} \subset \mathbf{G}_m^{n}$ constitutes an algebraic subvariety. We observe the relationship:

\begin{equation}
\operatorname{Stab}_{\mathcal{B}_{\eta}}(X_{\eta}) = \operatorname{Stab}_{\mathcal{B}_{\eta}}((\mathcal{X}_{r,d})_{\eta})^{u}.
\end{equation}

Finally, let $\beta_{M} := \alpha_{u,M}|_{X_{\eta}^{M+1}}$ denote the restriction of $\alpha_{u,M}$ to $X_{\eta}^{M+1}$.

\begin{lem}\label{lem:stab_orbits}
For $M$ sufficiently large, the generic fibers of the map $\beta_{M}: X_{\eta}^{[M+1]} \rightarrow (\mathbf{G}_m^t)^{M}$ are orbits under the stabilizer subgroup $\operatorname{Stab}_{\mathcal{B}_{\eta}}(X_{\eta})$.
\end{lem}

\begin{proof}[Proof of the lemma]
For any tuple $(x_1,\ldots,x_{M+1}) \in X^{M+1}$, define 
\begin{equation}
G(x_1,\ldots,x_{M+1}) := \left\{ b \in \mathbf{G}_m^t \mid bx_i \in X \text{ for all } i=1,\ldots,M+1 \right\}.
\end{equation}
The fiber of $\beta_M$ containing $(x_1,\ldots,x_{M+1})$ consists precisely of
\begin{equation}
\left\{ (bx_1,\ldots,bx_{M+1}) \mid b \in G(x_1,\ldots,x_{M+1}) \right\}.
\end{equation}

Notice that the intersection of finitely many subvarieties of the form $G(x_1,..., x_{M+1})$ is still the form $G(y_1,..., y_{N})$. The intersection of all subvarieties of the form $G(x_1,..., x_{M+1})$ is $\operatorname{Stab}_{\mathcal{B}_{\eta}}(X_{\eta})$.

Therefore, there exists $M' \in \mathbb{N}$ and a tuple $(x_1,\ldots,x_{M'+1}) \in X_{\eta}^{M'}$ such that
\begin{equation}
G(x_1,\ldots,x_{M'+1}) = \operatorname{Stab}_{\mathcal{B}_{\eta}}(X_{\eta}).
\end{equation}
The result follows by taking any integer $M \geq M'$.
\end{proof}

Observe that $\operatorname{Stab}_{\mathcal{B}_{\eta}}(X_{\eta})$ is finite by the defining properties of $S'$. Let us define
\begin{equation}
G_{\eta} := \{0\} \times \operatorname{Stab}_{\mathcal{B}_{\eta}}(X_{\eta}) \subset \mathcal{B}_{\eta}^{u} \times \mathcal{B}_{\eta}^{u(M+1)} = \mathcal{B}_{\eta}^{u(M+2)},
\end{equation}
where we mildly abuse notation by embedding $\operatorname{Stab}_{\mathcal{B}_{\eta}}(X_{\eta})$ diagonally into $\mathcal{B}_{\eta}^{u(M+1)}$. The preceding lemma demonstrates that the generic fiber of $\beta_M$ coincides with a $G_{\eta}$-orbit.

Choose a Zariski open subset $U \subset S'$ where $G_{\eta}$ extends to a finite group subscheme of $\mathcal{B}_{S'}^{[u(M+2)]}|_{U}$. Subsequently, select a Zariski open dense subset $V$ of $\alpha(X^{[M+2]})|_{U}$ such that:
\begin{itemize}
\item The restriction $\alpha: \alpha^{-1}(V) \rightarrow V$ becomes finite étale;
\item For each $y \in V(\mathbb{C})$ with $s = \pi(y)$, the fiber $\alpha^{-1}(y)$ forms a $G_{s}$-orbit.
\end{itemize}
This establishes that $\alpha$ becomes generically finite when $M$ is sufficiently large.

For such $M$, the equidimensionality $\dim X^{[M+2]} = \dim \alpha(X^{[M+2]}) = d_M$ follows from the generic finiteness of $\alpha$. Moreover, we can choose $M$ large enough such that $d_M > t$.

We will show that\begin{equation}
    \alpha_*\mu_1 \neq \mu_2.
\end{equation} 
Recall from last section the explicit descriptions of equilibrium measures $\mu_1$ and $\mu_2$. Let $I$ index the $d_M$-tuples $\mathbf{i} = (i_1,\ldots,i_{d_M})$ where $M_{i_1} \wedge \cdots \wedge M_{i_{d_M}}$ doesn't vanish identically on $\overline{X^{[M+2]}}$. Recall
\begin{equation}
X_{\mathbf{i}} := \{x \in \overline{X^{[M+2]}} : |\phi_{i_1}(x)| = \cdots = |\phi_{i_{d_M}}(x)| = 1\}
\end{equation}
\begin{equation}
\phi_{\mathbf{i}} := (\phi_{i_1},\ldots,\phi_{i_{d_M}}): (\mathbf{G}_m)^{t(M+2)} \times \mathbf{H} \rightarrow (\mathbb{P}^1)^{t(M+2)}
\end{equation}
\begin{equation}
E_{\mathbf{i}} := \left\{x \in \overline{X^{[M+2]}} : \ker(d\phi_{\mathbf{i}}|_{\overline{X^{[M+2]}}}) \neq \{0_x\}\right\} \cup \left((\mathbb{P}^1)^{t(M+2)} \setminus \mathbf{G}_m^{t(M+2)}\right)
\end{equation}
\begin{equation}
dw_{\mathbf{i}} := \phi_{\mathbf{i}}^*dw
\end{equation}
The measure $\mu_1$ decomposes as
\begin{equation}
\int_{\overline{X^{[M+2]}}} f d\mu_1 = k \sum_{\mathbf{i} \in I} \int_{X_{\mathbf{i}} \setminus E_{\mathbf{i}}} f dw_{\mathbf{i}}
\end{equation}
for some $k > 0$ and all $f \in C_c(\overline{X^{[M+2]}})$.

Analogously, for $\alpha(X^{[M+2]})$ with index set $J$, define
\begin{equation}
Y_{\mathbf{j}} := \{x \in \overline{\alpha(X^{[M+2]})} : |\phi_{j_1}(x)| = \cdots = |\phi_{j_{d_M}}(x)| = 1\}
\end{equation}
\begin{equation}
F_{\mathbf{j}} := \left\{x \in \overline{\alpha(X^{[M+2]})} : \ker(d\phi_{\mathbf{j}}|_{\overline{\alpha(X^{[M+2]})}}) \neq \{0_x\}\right\} \cup \left((\mathbb{P}^1)^{t(M+2)} \setminus \mathbf{G}_m^{t(M+2)}\right)
\end{equation}
yielding the measure decomposition
\begin{equation}
\int_{\overline{\alpha(X^{[M+2]})}} g d\mu_2 = k' \sum_{\mathbf{j} \in J} \int_{Y_{\mathbf{j}} \setminus F_{\mathbf{j}}} g dw_{\mathbf{j}},
\end{equation}
for some constant $k'$.
Define 
\begin{equation}
E := \{x \in X^{[M+2]} \mid \ker(d\alpha|_{X^{[M+2]}}) \neq \{0_x\}\}.
\end{equation}
The generic finiteness of $\alpha$ ensures $X^{[M+2]} \setminus E \neq \emptyset$, making $E$ a locally pluripolar subset as it is a complex analytic subset. Since $\overline{\alpha(E)}$ is pluripolar and $\mu_2$ doesn't assign any mass on pluripolar sets (it is a fundamental fact in pluripotential theory), we may assume
\begin{equation}
Y_{\mathbf{j}} \setminus F_{\mathbf{j}} \subset \alpha(X^{[M+2]} \setminus E)
\end{equation}
with the measure relation
\begin{equation}
\int_{Y_{\mathbf{j}} \setminus F_{\mathbf{j}}} f dw_{\mathbf{j}} = \int_{\alpha^{-1}(Y_{\mathbf{j}} \setminus F_{\mathbf{j}})} (f \circ \alpha) \alpha^* dw_{\mathbf{j}}.
\end{equation}

If $\alpha_*\mu_1 = \mu_2$ held, we would have the equality
\begin{equation}
\overline{\bigcup_{\mathbf{i} \in I} (X_{\mathbf{i}} \setminus E_{\mathbf{i}})} = \overline{\bigcup_{\mathbf{j} \in J} \alpha^{-1}(Y_{\mathbf{j}} \setminus F_{\mathbf{j}})}.
\end{equation}
The right-hand side contains the subset $((S^1)^t \times \mathbb{P}^w) \cap \overline{X}\times \Delta$ where $\Delta$ is the diagonal in $X^{[M+1]}$.

However, since $d_M > t$ and each $\mathbf{i}$ indexes $d_M$ coordinates, any point $$x = (x_1,\ldots,x_{t(M+2)},z) \in X^{[M+2]} \subset (\mathbb{P}^1)^{t(M+2)} \times \mathbb{P}^w$$ in the left-hand side must have at least $t(M+1)$ coordinates in $S^1$. This directly contradicts Lemma \ref{5521}(a), completing the proof by contradiction.

\textbf{Step 3. Apply equidistribution result twice.}

As $\alpha_*\mu_1 \neq \mu_2$, we may 
select test functions $f_1 \in C_c^\infty(X^{[M+2],\mathrm{an}})$ and $f_2 \in C_c^\infty(\alpha(X^{[M+2]})^{\mathrm{an}})$ satisfying:
\begin{enumerate}
\item Existence of $\epsilon_{S'} > 0$ such that
\begin{equation}\label{C}
\left| \int_{X^{[M+2],\mathrm{an}}} f_1 d\mu_1 - \int_{\alpha(X^{[M+2]})^{\mathrm{an}}} f_2 d\mu_2 \right| \geq 2\epsilon_{S'}
\end{equation}

\item Function relation: $f_1 = f_2 \circ \alpha$
\end{enumerate}

\begin{lem}\label{sh}
There exist $\delta_{S'} > 0$ and a proper Zariski closed subset $Z_{S'} \subset X^{[M+2]}$ such that for all $x \in (X^{[M+2]} \setminus Z_{S'})(\overline{\mathbb{Q}})$:
\begin{enumerate}
\item Either $\hat{h}_{L^{\boxtimes u(M+2)}}(x) \geq \delta_{S'}$,
\item Or $\hat{h}_{L^{\boxtimes u(M+1)}}(\alpha(x)) \geq \delta_{S'}$.
\end{enumerate}
\end{lem}

\begin{proof}[Proof of the lemma]
Using the test functions $f_1,f_2$ as above, apply Proposition \ref{433} to the triples $(X^{[M+2]}, f_1, \epsilon_{S'})$ and $(\alpha(X^{[M+2]}), f_2, \epsilon_{S'})$.

This produces a constant $\delta_{S'} > 0$,, proper Zariski closed subsets $Z_{S',1} \subset X^{[M+2]}$ and $Z_{S',2} \subset \alpha(X^{[M+2]})$ satisfying Proposition \ref{433}. Define the exceptional set
\begin{equation}
Z_{S'} := Z_{S',1} \cup \alpha^{-1}(Z_{S',2}),
\end{equation}
which remains a proper Zariski closed subset of $X^{[M+2]}$.

For any $x \in (X^{[M+2]} \setminus Z_{S'})(\overline{\mathbb{Q}})$ with both
\begin{align*}
\hat{h}_{(L')^{\boxtimes u(M+2)}}(x) &< \delta_{S'} \quad \text{and} \\
\hat{h}_{L^{\boxtimes u(M+1)}}(\alpha(x)) &< \delta_{S'},
\end{align*}
the equidistribution estimates yield:
\begin{equation}
\left| \frac{1}{\# O(x)} \sum_{y \in O(x)} f_1(y) - \int_{X^{[M+2],\mathrm{an}}} f_1 d\mu_1 \right| < \epsilon_{S'}
\end{equation}
\begin{equation}
\left| \frac{1}{\# O(\alpha(x))} \sum_{y \in O(\alpha(x))} f_2(y) - \int_{\alpha(X^{[M+2]})^{\mathrm{an}}} f_2 d\mu_2 \right| < \epsilon_{S'}
\end{equation}

From the functional relation $f_1 = f_2 \circ \alpha$, we derive:
\begin{equation}
\frac{1}{\# O(x)} \sum_{y \in O(x)} f_1(y) = \frac{1}{\# O(\alpha(x))} \sum_{y \in O(\alpha(x))} f_2(y)
\end{equation}

This leads to the contradiction:
\begin{equation}
\left| \int_{X^{[M+2],\mathrm{an}}} f_1 d\mu_1 - \int_{\alpha(X^{[M+2]})^{\mathrm{an}}} f_2 d\mu_2 \right| < 2\epsilon_{S'},
\end{equation}
directly opposing inequality (\ref{C}).
\end{proof}

\textbf{Step 4. Induction to conclude.}

We establish the result by induction on \( v := \dim S \). Let \( S'' := S' \setminus \pi(X^{[M+2]} \setminus Z_{S'}) \) with reduced subscheme structure. Since \( X^{[M+2]} \setminus Z_{S'} \) is non-empty, we have:
\[
\dim S'' \leq \dim S' - 1
\]
The construction of \( S'' \) depends solely on \( r,d,n,S' \), with finitely many irreducible components determined by these parameters. Replace by $S''$ by its irreducible component, we may assume \( S'' \) is irreducible without loss of generality.

Suppose $s\in S^{'}(\overline{\mathbb{Q}})$ be a point satisfying (1) and (2) in the statement of the theorem. 

\begin{itemize}
\item If \( s \in S''(\overline{\mathbb{Q}}) \), induction provides constants \( c_1(r,d,n,S''), c_2(r,d,n,S'') \) and a proper Zariski closed \( X' \subset (\mathcal{X}_{r,d})_s \) such that:
\begin{enumerate}
\item \( \{x \in (\mathcal{X}_{r,d})_s(\overline{\mathbb{Q}}) : \hat{h}_L(x) \leq c_2\} \subset X' \)
\item \( \deg_{L_s}(X') < c_1 \)
\end{enumerate}
As $S^{''}$ only depends on $S$, constants $c_1(r,d,n,S^{''})$, $c_2(r,d,n,S^{''})$ depend on $r,d,n,S$ solely (but not $s\in S(\overline{\mathbb{Q}})$. ).

\item If \( s \notin S''(\overline{\mathbb{Q}}) \), define \( c'_2 := \delta_{S'}/(4u(M+2)) \) and consider:
\[
\Sigma_{S',s} := \{x \in (\mathcal{X}_{r,d})_s(\overline{\mathbb{Q}}) : \hat{h}_L(x) \leq c'_2\}
\]
\end{itemize}

\textbf{Claim:} \( \Sigma_{S',s}^{u(M+2)} \subset (Z_{S'})_s(\overline{\mathbb{Q}}) \)

For \( x = (x_1,\ldots,x_{u(M+2)}) \in \Sigma_{S',s}^{u(M+2)} \setminus (Z_{S'})_s(\overline{\mathbb{Q}}) \):
\[
\hat{h}_{L^{\boxtimes u(M+2)}}(x) = \sum_{i=1}^{u(M+2)} \hat{h}_L(x_i) \leq u(M+2)c'_2 < \delta_{S'}
\]
\[
\hat{h}_L(x_ix_j^{-1}) \leq 2\hat{h}_L(x_i) + 2\hat{h}_L(x_j) \leq 4c'_2 \implies \hat{h}_{L^{\boxtimes u(M+1)}}(\alpha(x)) \leq 4u(M+1)c'_2 < \delta_{S'}
\]
This contradicts Lemma \ref{sh}.

\begin{lem}[Lemma 4.3, \cite{Gao21}]
Let $k$ be a field, $Y$ an irreducible projective variety over $k$, and $L^{'}$ a very ample line bundle on $Y$. Consider a closed subvariety $Z^{'}\subsetneq Y^{N}$. Then, there exists a constant
$$c=c(N,\dim Y,\deg_{L^{'}}Y,\deg_{(L^{'})^{\boxtimes N}}(Z^{'}))>0$$ such that, for any subset $\Sigma\subset Y(k)$ satisfying $\Sigma^{N}\subset Z^{'}(k)$ there exists a proper closed subvariety $Y^{'}\subset Y$ with $\Sigma\subset Y^{'}(k)$ and $\deg_{L^{'}}(Y^{'})<c$.
\end{lem}

Invoke this lemma with parameters:
\begin{itemize}
\item \( Y := (\mathcal{X}_{r,d})_s \), \( L' = L_s|_{(\mathcal{X}_{r,d})_s} \)
\item \( Z' = Z_{S',s} \), \( \Sigma = \Sigma_{S',s} \)
\end{itemize}
This produces a proper Zariski closed \( X' \subset (\mathcal{X}_{r,d})_s \) satisfying:
\begin{enumerate}
\item \( \Sigma_{S,s} \subset X'(\overline{\mathbb{Q}}) \)
\item \( \deg_{L_s}(X') \leq c(u(M+2), r, d, \deg_{L_s^{\boxtimes u(M+2)}}(Z_{S',s})) \)
\end{enumerate}

Since \( u, M, Z_{S'} \) depend only on \( r,d,n,S \), the bounding constant:
\[
c(u(M+2), r, d, \deg_{L_s^{\boxtimes u(M+2)}}(Z_{S',s}))
\]
ultimately depends only on \( r,d,n,S \), completing the proof.
\end{proof}

\Addresses
\end{document}